\documentclass{amsart}
\usepackage[utf8]{inputenc}
\usepackage{amsthm}
\usepackage{amsmath}
\usepackage{amssymb}
\usepackage{enumitem}
    \setlist[enumerate,1]{label=(\alph*), font=\normalfont}
    \setlist[enumerate,2]{label=(\roman*), font=\normalfont}
    \setlist{nosep}
\usepackage{tikz}
\usepackage{xcolor}
 \usepackage{subcaption}
\usepackage{todonotes}

\newcommand{\qf}[1]{\langle #1\rangle}

\usepackage[
backend=biber,
style=alphabetic,
sorting=nyt,
isbn=false,
url=false,
]{biblatex}
\usepackage{hyperref}
\usepackage{cleveref}

\usepackage{multirow}

\newtheorem{definition}{Definition}[section]

\newtheorem{theorem}[definition]{Theorem}
\newtheorem{lemma}[definition]{Lemma}
\newtheorem{corollary}[definition]{Corollary}
\newtheorem{conjecture}[definition]{Conjecture}
\newtheorem{remark}[definition]{Remark}
\newtheorem{example}[definition]{Example}
\newtheorem{proposition}[definition]{Proposition}

\newcommand{\N}{\mathbb{N}}
\newcommand{\F}{\mathbb{F}}
\newcommand{\Q}{\mathbb{Q}}
\newcommand{\bH}{\mathbb{H}}

\newcommand{\Gcal}{\mathcal{G}}

\DeclareMathOperator{\cha}{char }

\newcommand{\aut}{\operatorname{aut}}
\newcommand{\og}{\mathsf{O}}
\newcommand{\iso}{\operatorname{iso}}

\newcommand{\spn}{\operatorname{span}}

\newcommand{\dist}{\mathsf{d}}
\newcommand{\diam}{\mathsf{diam}}
\newcommand{\rad}{\mathrm{rad}}

\newcommand{\D}{\mathsf D}
\newcommand{\Dtilde}{\widetilde{\mathsf D}}
\newcommand{\G}{\mathsf G}

\newcommand{\g}{\mathsf g}
\newcommand{\iw}{\mathsf i_W}


\title{Diameter and girth of Representation Graphs of Quadratic Forms}
\author{Nico Lorenz}

\address{N.~Lorenz, Fakult\"at für Mathematik, Ruhr Universit\"at Bochum, Universit\"atsstra{\ss}e 150, 44780 Bochum, Germany}
\email{nico.lorenz@ruhr-uni-bochum.de}

\author{Marc Christian Zimmermann}

\address{M.C.~Zimmermann, Department Mathematik/Informatik, Abteilung
  Mathematik, Universit\"at zu K\"oln, Weyertal~86--90, 50931 K\"oln,
  Germany}
  \email{marc.christian.zimmermann@gmail.com}
\date{\today}

\begin{document}

\begin{abstract}
    Let $q$ be a non-degenerate quadratic form defined on an $F$ vector space $V$ and $a \in F$.
    We consider the Cayley graph on $V$ with generating set $\{x \in V \mid q(x) = a\}$ and study its diameter and girth.
    In particular, if $F$ is a finite field, we calculate these invariants and the number of cycles of minimal length in these graphs.
\end{abstract}

\maketitle

\section{Introduction}

Quadratic forms are classical objects that appear in many different areas of mathematics.
Originating in number theory, the algebraic theory of quadratic forms initiated by the work of E. Witt in \cite{MR1581519}, has connections to Galois theory, central-simple algebras, and cohomology, among others.
Furthermore, quadratic forms induce combinatorial objects which are regularly studied in the literature and which will be investigated in this article.
Given a quadratic form $q$ on a vector space $V$ over a field $F$ and a fixed element $a \in F$, one can consider the \emph{representation graph} of $q$ with respect to $a$, i.e. the graph with vertex set $V$ in which two different vertices $x, y \in V$ are connected if and only if $q(x - y) = a$.
In other words, we consider the Cayley graphs on $V$ with generating set $V_{q, a} = \{x \in V \setminus \{0\} \mid q(x) = a\}$.
A special emphasis in the literature lies on the \emph{sum of squares form}, i.e. the form defined on $F^n$ which maps a vector $(x_1, \ldots, x_n) \in F^n$ to $x_1^2 + \ldots + x_n^2$ and $a = 1$, since this situation can be interpreted as an analog of the Euclidean space of dimension $n$.
The articles mentioned in the following short survey all focus on the sum of squares form over different fields if not stated otherwise.

For the field of rational numbers $F = \Q$ these graphs are studied in \cite{Chilakamarri_1988} with respect to their connectedness and clique number.
From a combinatorial point of view, representation graphs are of particular interest in the case of a finite field (or even ring).
The spectral properties (also implying connectedness results) of representation graphs over finite fields were obtained in \cite{medrano_myers_stark_terras_1996}.
Formulated in the language of association schemes, generalizations of these results, also for arbitrary non-degenerate quadratic forms, can be found in \cite{bannai_shimabukuro_tanaka_2009}.

The connectedness in the case of finite fields was also derived using combinatorial techniques in a recent paper by M. Krebs in \cite{Krebs_2022}.
This article also deals with clique numbers and includes some results on the chromatic number.
In \cite{lorenz:2023:cir}, the authors reprove the results of \cite{Krebs_2022} about the clique number using quadratic form theory.
Using this algebraic approach, the results could be generalized to the case of representation graphs of arbitrary non-degenerate quadratic forms and further, the number of maximal cliques instead of just their size could be computed.

Similarly the main goal of the present article is to generalize and quantify the results on connectedness.
While we focus on finite fields, some of our results are also valid for infinite fields.
We expand the results on the connectedness by calculating the \emph{diameter}, i.e. the maximal length of minimal paths between two vertices is calculated in almost all cases over finite fields, see \Cref{thm:diameter:finite:fields}.
Similar techniques as for the diameter are used to compute the \emph{girth}, i.e. the length of a shortest cycle, see \Cref{thm:girth}.
Finally, the number of shortest cycles is calculate in \Cref{prop:CyclesLength3} and \Cref{thm:cycles4anisotropic}.

As noted above, our techniques differ from those in the literature in that we rely heavily on results from the algebraic theory of quadratic forms.
This theory provides results that can be used in combinatorial arguments.
Furthermore, the algebraic arguments can be used to reduce the general case of a given problem to a special case in which combinatorial arguments work more easily.
As one of the main tools, we use decompositions of quadratic forms.
Results on low-dimensional forms, which are naturally easier to establish, can then easily be used to cover the arbitrary case.
The final advantage of our approach is that some of our results also hold over infinite fields.

The structure of the paper is as follows. 
In \Cref{sec:prelim}, we recall some basic concepts about graphs, quadratic forms and their representation graphs.
In \Cref{sec:Connectedness}, we deal with connectedness and diameter of representation graphs over finite fields and even for arbitrary fields under some further conditions on the quadratic form.
Finally, \Cref{sec:girth} is devoted to the length and the number of cycles of minimum length.

\section{Preliminaries} \label{sec:prelim}

\subsection{Quadratic Forms}

We briefly recall some basic facts about and notation on quadratic forms that will be used frequently in this article.
We use standard terminology when working with quadratic forms, usually consistent with that used in \cite{MR2427530} and most of the results on quadratic forms we use can be found in \cite[Sections 7, 8, 13]{MR2427530}.
Readers familiar with the theory of quadratic forms may want to skip this section.

Let $F$ be a field and $V$ be a finite dimensional vector space over $F$.
A \emph{quadratic form} on $V$ is a map $q: V \to F$ such that
    \begin{enumerate}
        \item[(QF1)] for all $a\in F$ and $x\in V$ we have $q(ax) = a^2 q(x)$ and
        \item[(QF2)] the map $b_q:V\times V\to F, (x,y) \mapsto q(x + y) - q(x) - q(y)$ is bilinear.
    \end{enumerate}
In particular we have $b_q(x,x) = 2q(x)$.
Throughout the text, when a quadratic form $q$ is present, we refer to $V$ as the underlying vector space of $q$ and as usual in quadratic form theory, we write $\dim(q)$ for $\dim(V)$.

For example, for $a_1, \ldots, a_n \in
F$, a quadratic form on $F^n$ is given by
\[(x_1, \ldots, x_n) \mapsto a_1x_1^2 + \ldots + a_nx_n^2\]
and will be denoted by $\qf{a_1,\ldots, a_n}$.
For $a,b\in F$, the quadratic form $F^2 \to F, (x,y) \mapsto ax^2 + xy + by^2$ is denoted by $[a,b]$.
Orthogonal sums are denoted by $\perp$ and isometries are denoted by $\cong$.

Our focus is on \emph{non-degenerate} quadratic forms, i.e., forms with 
\[\rad(b_q) = \{x \in V \mid \text{ for all $y \in V$: } b_q(x, y) = 0  \} = \{0\}.\]
Note that this convention differs from our main reference \cite{MR2427530}, in which non-degenerate quadratic forms are defined in a slightly more general way (including forms that are often called \emph{semisingular}).
The more general definition used there comes from a geometric approach to quadratic form theory, then a quadratic form is non-degenerate if and only if the underlying quadric is smooth (see \cite[Proposition 22.1]{MR2427530}).
However, since we are focussing on algebraic aspects, we use the more classical algebraic definition of non-degenerate\footnote{Note that the term \emph{non-degenerate} is used inconsistently in the literature. 
Other terms that are regularly used are \emph{non-singular} or \emph{regular} and sometimes, these expressions may also refer to weaker conditions.} quadratic forms. \medskip

\emph{Unless explicitly stated otherwise, all quadratic forms are assumed to be non-degenerate.}\medskip

Extending the usual definition of represented elements 
\[\D(q) = \{q(v) \in F^\ast \mid v \in V\}\]
we further use the definition
\[\Dtilde(q) = \begin{cases}
    \D(q), &\text{if $q$ is anisotropic}\\
    \D(q) \cup \{0\},&\text{if $q$ is isotropic}
\end{cases},\]
where $q$ is \emph{isotropic} if there exists a nonzero vector $v$ with $q(v) =0$, and \emph{anisotropic} otherwise.

Recall that the \emph{hyperbolic plane} $\bH = [0, 0]$ is, up to isometry, the only non-degenerate quadratic form of dimension 2 which is isotropic.\medskip

Since our combinatorial interest focuses on finite structures, we have a particular interest in quadratic forms over finite fields, which are well understood.
We give a complete classification of quadratic forms over finite fields.
For this, recall the definition of the \emph{Artin-Schreier map} $\wp:F \to F, x \mapsto x^2 + x$ for fields of characteristic 2.

\begin{theorem}[{\cite[Chapter 2, Theorems 3.3, 3.8; Chapter 9, 4.5 Theorem]{book:Scharlau}}]  \label{thm:ClassificationQFFIniteFields}
    Let $F$ be a finite field, $n \in \N$ a positive integer and $k = \lfloor \frac n 2 \rfloor$.
    \begin{enumerate}
        \item Let $\cha(F) \neq 2$.
        Then we have $|F^\ast / F^{\ast2}| = 2$.
        Let $\lambda \in F^\ast \setminus F^{\ast2}$ be not a square.
        There are exactly two different quadratic forms up to isometry of dimension $n$.
        If $n$ is odd, they are given by 
        \[k \times \bH \perp \qf1 \quad \text{ and } \quad k \times \bH \perp \qf \lambda.\]
        If $n$ is even, these are given by 
        \[k \times \bH \quad \text{ and } \quad (k - 1) \times \bH \perp \qf {1, -\lambda}.\]
        \item Let $\cha(F) = 2$.
        Then we have $|F / \wp(F)| = 2$.
        Let $\lambda \in F \setminus \wp(F)$.
        If $n$ is odd, there are no quadratic forms of dimension $n$.
        If $n$ is even, there are exactly two different quadratic forms up to isometry of dimension $n$.
        These are given by 
        \[k \times \bH \quad \text{ and } \quad (k - 1) \times \bH \perp [1, \lambda].\]
    \end{enumerate}
\end{theorem}

Recall that an element $a \in F^\ast$ is called a \emph{similarity factor} of a quadratic $q$ if $aq \cong q$.
The set of similarity factors forms a group and is denoted by $\G(q)$.
A quadratic form is called \emph{round}\footnote{This term was introduced by the German mathematician A. Pfister as a pun: the German word \emph{quadratisch} translates to both \emph{quadratic} and \emph{square}, so that \emph{runde quadratische Form} can be translated to \emph{round square}} if $\D(q) = \G(q)$.
For example, the hyperbolic plane $\bH$ is easily seen to be round with $\D(\bH) = \G(\bH) = F^\ast$.
The classification of quadratic forms over finite fields easily implies the following well-known consequences, which we will recall for the convenience of the reader.

\begin{remark} \label{rem:finite_field_remarks}
    Let $F$ be a finite field and $q$ be a quadratic form over $F$.
    \begin{enumerate}
        \item \label{rem:finite_field_remarks1} The isometry type of $q$ over $F$ can be determined just by its dimension and its determinant resp. Arf invariant, depending on whether $\cha(F) \neq 2$ or $\cha(F) = 2$.
        \item \label{rem:finite_field_remarks2} If $\dim(q) \geq 2$, then $q$ is universal, i.e. $\D(q) = F^\ast$.
        \item \label{rem:finite_field_remarks3} If $\dim(q) = 2$, then $q$ is either hyperbolic or anisotropic.
        \item \label{rem:finite_field_remarks4} 
        If $\dim(q)$ is even, then $q$ is round.
        If $\dim(q)$ is odd, we have $\G(q) = F^{\ast2}$ (the latter equality holds over all fields as is easily seen by considering the determinant).
    \end{enumerate}
\end{remark}

Furthermore, the computation of some numerical invariants for quadratic forms over finite fields are very classical results.
In particular, the size and structure of the \emph{orthogonal group}
\[\og(q) = \{t: V \to V \mid q \text{ is an isometry of $q$} \}\]
of a quadratic form $q$ is well known (see \cite[Sections 3.7.2, 3.8.2]{book:Wilson}).

Over a finite field, it is also known in how many ways a given element is represented by a quadratic form.
We write
\[V_{q, a} = \{ x \in V \ :\ q(x) = a \},\]
where we drop $q$ from the subscript if $q$ is clear from the context.
We care about the cardinality of these sets, in particular for $2$-dimensional forms.
We include a full description for the convenience of the reader.

\begin{proposition}[{\cite[6.26, 6.27, 6.32 Theorem]{book:lidl}}] \label{prop:NumberOfRepresentations}
    Let $q$ be a quadratic form defined on a vector space $V$ over a field $F$ of finite cardinality $|F| = f$.
    Let $a\in F^\ast$.
    The number of vectors $v \in V$ such that $q(v) = a$ resp. $q(v) = 0$ is given as in \Cref{tbl:NUmberOfRepresentations}, where the $+$ sign in the lower right cell applies if and only if $(-1)^m \cdot \det(q) = a$.
    \begin{table}[!ht]
        \renewcommand{\arraystretch}{1.25}
        \begin{tabular}{|c||c|c|}
            \hline
            type of quadratic form & $|V_0|$ & $|V_a|$\\  \hline\hline
            $\dim(q) = 2m$, $q$ hyperbolic & $f^{2m - 1} + f^m - f ^{m - 1}$ & $f^{2m - 1} - f ^{m - 1}$  \\ \hline
            $\dim(q) = 2m$, $q$ not hyperbolic & $f^{2m - 1} - f^m + f^{m - 1}$ & $f^{2m - 1} + f^{m - 1}$  \\ \hline 
            $\dim(q) = 2m + 1$ & $f^{2m}$ & $f^{2m} \pm f^m$  \\ \hline
        \end{tabular}
        \caption{Size of the pre-image under a quadratic form} \label{tbl:NUmberOfRepresentations}
        \end{table}
\end{proposition}

Of particular importance for us is the fact that the orthogonal group $\og(q)$ acts transitively on vectors representing the same element; this is a direct consequence of the fundamental Witt Extension Theorem.

\begin{corollary} \label{cor:OGtransitive}
    Let $q$ be a quadratic form defined on a vector space $V$ and let $v,w\in V\setminus\{0\}$ with $q(v) = q(w)$.
    Then there is some $\sigma\in \og(q)$ with $\sigma(v) = w$.
\end{corollary}
\begin{proof}
    Since $\rad(b_q) = \{0\}$, we can apply Witt's Extension Theorem \cite[Theorem 8.3]{MR2427530} to extend the isometry $q|_{\spn(v)} \to q|_{\spn(w)}$ defined by $\lambda v \mapsto \lambda w$ to such a $\sigma \in \og(q)$ as desired. 
\end{proof}

\subsection{Graph Theory}

We now recall some standard notation for graphs and provide some basic properties of representation graphs of quadratic forms.

Let $G = (V,E)$ be a (not necessarily finite, but simple) graph, where $V$ is the \emph{vertex set} and $E \subseteq \binom{V}{2}$ the \emph{edge set} of $G$. 
We say that two vertices $u,v \in V$ are \emph{adjacent} in $G$ if $\{u,v\} \in E$. 
A \emph{path of length $n$} between two vertices $u, v \in V$ is a sequence $(v_0 = u, v_1, \ldots, v_n = v)$ of vertices such that $v_i \neq v_j$ for $i \neq j$ and $\{v_{i - 1}, v_{i}\} \in E$ for all $i \in \{1, \ldots, n\}$.
Similarly, a \emph{cycle} of length $n \geq 3$ is a sequence $(v_0, v_1, \ldots, v_{n - 1}, v_0)$ of vertices such that $v_i \neq v_j$ for $i \neq j$ and $\{v_{i - 1}, v_{i}\} \in E$ for all $i \in \{1, \ldots, n\}$.

We define the \emph{distance} $\dist(u,v)$ of two vertices $u, v \in V$ as the length of a shortest $u$-$v$-path in $G$ and $\infty$ if no such path exists, the \emph{diameter} of $G$ by 
\[\diam(G) = \sup\{n\in \N \mid \text{there are } u,v \in V \text{ such that } d(u,v) = n\} \in \N_0 \cup \{\infty\},\]
and the \emph{girth} of $G$ by 
\[\g(G) = \inf\{n\in \N \mid \text{there is a cycle of length $n$ in $G$}\} \in \N_{\geq 3} \cup \{\infty\}.\]

Two graphs $G = (V, E), G' = (V', E')$ are \emph{isomorphic} if and only if there exists a bijective map $\rho: V \rightarrow V'$ such that $\{\rho(u), \rho(v)\} \in E'$ if and only if $ \{u,v\} \in E$, we write $G \cong G'$ in this case.
It is clear that diameter and girth are invariants of the isomorphism class of a graph.
The group 
\[\aut(G) = \{ \sigma: V \to V \ \mid \sigma\text{ is an isomorphism of $G$}\}\] 
is the \emph{automorphism group} of the graph $G$. \medskip

We now introduce a family of graphs for a given quadratic form.
These graphs are the protagonists of this article.

\begin{definition}
	Let $q$ be a quadratic form on an $F$ vector space $V$ and $a\in F$. 
    The \emph{representation graph} $\mathcal G_{q, a} = (V, E)$ is defined by its set of edges
	\[E:=\left\{ \{x,y\} \in \binom{V}{2} \mid q(x - y) = a \right\}.\]
\end{definition}

It is not hard to see that a  representation graph $G_{q, a} = (V, E)$ is the \emph{Cayley graph} $\operatorname{Cay}(V, V_{q, a})$ for the group $(V, +)$ and the symmetric set 
\[V_{q, a} = \{ x \in V \ \mid q(x) = a \},\]
i.e. we have $\{x, y\} \in E$ if and only if $x - y \in V_{q, a}$.

The following lemma is well known.

\begin{lemma} \label{lem:cayley:connected}
    The number of connected components of a Cayley graph $\operatorname{Cay}(V; S)$ is equal to the index $[V:\qf S]$, where $\qf S$ denotes the subgroup generated by the elements of $S$.
    In particular $\operatorname{Cay}(V, S)$ is connected if and only if $S$ generates $V$ as a group.
\end{lemma}

In the context of representation graphs it is natural to ask how isomorphy of graphs relates to isometry of quadratic forms. 
We first note that it is clear that if $q \cong q'$, then we also have $\Gcal_{q, a} \cong \Gcal_{q', a}$ for all $a \in F$.

We also consider the following extension of isometries.
The \emph{translation} with \emph{translation vector} $v$ is the map 
\[\tau_v: V\to V, x \mapsto x + v.\]
We define the group of \emph{affine isometries} $\iso(q)$ of $q$ as the subgroup of the affine group in which the linear part is an isometry of $q$, i.e. 
\[\iso(q) = \{\tau_v \mid v\in V\} \rtimes \og(q).\]
It is easy to see that every element of $\iso(q)$ is a graph automorphism of $\Gcal_{q, a} = (V, E)$.

Recall that an element $u \in F^\ast$ is called a \emph{similarity factor} for the quadratic form $q$ if we have $uq \cong q$. 
By \cite[Proposition 3.1]{lorenz:2023:cir}, we obtain the basic graph isomorphism $\Gcal_{q, a} \cong \Gcal_{q, \alpha a}$, whenever $\alpha \in F^\ast$ is a similarity factor of $q$.
We will use this fact implicitly throughout the text, particularly for quadratic forms $q$ of dimension $2$.
By \Cref{rem:finite_field_remarks} \ref{rem:finite_field_remarks4}, if $q \cong \bH$ or $q$ is defined over a finite field, then each $\alpha \in F^\ast$ is a similarity factor and so, for all $a \in F^\ast$, we have $\Gcal_{q, a} \cong \Gcal_{q, 1}$ in this case.

\section{Connectedness and diameter} \label{sec:Connectedness}

\subsection{First Examples} \label{sec:first:examples}

We start with some easy examples that exhibit some of the general behavior of representation graphs and illustrate that they range from being completely disconnected to being connected. 
Note that for fixed $q$ the graph whose edge set is the union of the edge sets of all graphs $\Gcal_{q,a}$, as $a$ exhausts $F$, is a complete graph.
This is related to the fact that the edge sets of the graphs $\Gcal_{q,a}$ are the relations of an association scheme on the underlying vector space (see \cite{bannai_shimabukuro_tanaka_2009}).

\subsubsection*{$1$-dimensional forms}
Let $q = \qf{b}$ be a quadratic form over a field $F$. If $ab \not\in F^{\ast 2}$, then the representation graph $\Gcal_{q,a}$ consists only of isolated points.

If $F= \F_{p^m}$ and $ab \in \F_p^{\ast 2}$, then $\Gcal_{q,a}$ is the disjoint union of $p^{m-1}$ cycles of length $p$, therefore it is connected if and only if $m=1$ (which is essentially already covered in \cite[Theorem 1.1]{Krebs_2022})

\subsubsection*{Hyperbolic planes over small fields}
Let $\bH = [0,0]$ be the hyperbolic plane over a field $F$. 
If $F$ is $\F_2$ or $\F_4$ and $a \neq 0$, the representation graph $\Gcal_{\bH,a}$ is not connected.
We illustrate this for $a=1$, where
\[
V_1 = \begin{cases} \{ (1,1) \} & \text{for $\F_2$} \\ \{ (1,1),(\alpha,\alpha+1),(\alpha+1,\alpha) \} & \text{for $\F_4$.} \end{cases}
\]
So for $\F_2$ the graph has 2 connected components, for $\F_4$ it has 4 connected components. 
The connected components of $(0,0)$ are shown in \Cref{fig:F2F4:hyp:plane:cc}.
Note that for both fields $\Gcal_{q,0}$ is connected (which is either a simple calculation or a consequence of the upcoming \Cref{prop:isotropic:diameter}).

For $\F_3$ and $\F_5$ we have
\[
V_1 = \begin{cases} \{ (1,1), (2,2) \} & \text{for $\F_3$} \\ \{ (1,1),(4,4),(2,3),(3,2) \} & \text{for $\F_5$.} \end{cases}
\]
From this we see that for $F = \F_3$ the graph $\Gcal_{\bH,1}$ has $3$ connected components while for $F = \F_5$ the graph $\Gcal_{\bH,1}$ is connected since $V_1$ generates $\F_5^2$ as a group.

Note that in all the examples above, $\Gcal_{\bH,a}\! \cong \Gcal_{\bH,1}$ for $a \neq 0$, since $a$ is a similarity factor of $\bH$.

\begin{figure}[h]

\begin{subfigure}{0.3\textwidth}
\centering
\begin{tikzpicture}[scale=0.7]

    \node at (0,0) {\textbullet} node[below] at (0,0) {\tiny $(0,0)$};
    \node at (0,2) {\textbullet} node[below] at (0,2) {\tiny $(0,1)$};

    \node at (2,0) {\textbullet} node[below] at (2,0) {\tiny$(1,0)$};
    \node at (2,2) {\textbullet} node[below] at (2,2) {\tiny$(1,1)$};

    \draw (0,0) -- (2,2);
\end{tikzpicture}
\caption{$\F_2$} \label{subfig:F2F4:hyp:plane:cc:a}
\end{subfigure}
\begin{subfigure}{0.6\textwidth}
\centering
\begin{tikzpicture}[scale=0.7]

    \node at (0,0) {\textbullet} node[below] at (0,0) {\tiny $(0,0)$};
    \node at (0,2) {\textbullet} node[below] at (0,2) {\tiny $(0,1)$};
    \node at (0,4) {\textbullet} node[below] at (0,4) {\tiny$(0,\alpha)$};
    \node at (0,6) {\textbullet} node[below] at (0,6) {\tiny$(0,\alpha\!+\!1)$};

    \node at (2,0) {\textbullet} node[below] at (2,0) {\tiny$(1,0)$};
    \node at (2,2) {\textbullet} node[below] at (2,2) {\tiny$(1,1)$};
    \node at (2,4) {\textbullet} node[below] at (2,4) {\tiny$(1,\alpha)$};
    \node at (2,6) {\textbullet} node[below] at (2,6) {\tiny$(1,\alpha\!+\!1)$};

    \node at (4,0) {\textbullet} node[below] at (4,0) {\tiny$(\alpha,0)$};
    \node at (4,2) {\textbullet} node[below] at (4,2) {\tiny$(\alpha,1)$};
    \node at (4,4) {\textbullet} node[below] at (4,4) {\tiny$(\alpha,\alpha)$};
    \node at (4,6) {\textbullet} node[below] at (4,6) {\tiny$(\alpha,\alpha\!+\!1)$};

    \node at (6,0) {\textbullet} node[below] at (6,0) {\tiny$(\alpha\!+\!1,0)$};
    \node at (6,2) {\textbullet} node[below] at (6,2) {\tiny$(\alpha\!+\!1,1)$};
    \node at (6,4) {\textbullet} node[below] at (6,4) {\tiny$(\alpha\!+\!1,\alpha)$};
    \node at (6,6) {\textbullet} node[below] at (6,6) {\tiny$(\alpha\!+\!1,\alpha\!+\!1)$};

    \draw (0,0) -- (2,2);
    \draw (0,0) -- (6,4);
    \draw (0,0) -- (4,6);
    \draw (2,2) -- (6,4);
    \draw (2,2) -- (4,6);
    \draw (4,6) -- (6,4);
\end{tikzpicture}

\caption{$\F_4$} \label{subfig:F2F4:hyp:plane:cc:b}
\end{subfigure}

\caption{The connected component of $(0,0)$ in the representation graph $\Gcal_{\bH,1}$ of the hyperbolic plane $\bH$ over (a): $\F_2$ and (b): $\F_4$.} \label{fig:F2F4:hyp:plane:cc}
\end{figure}
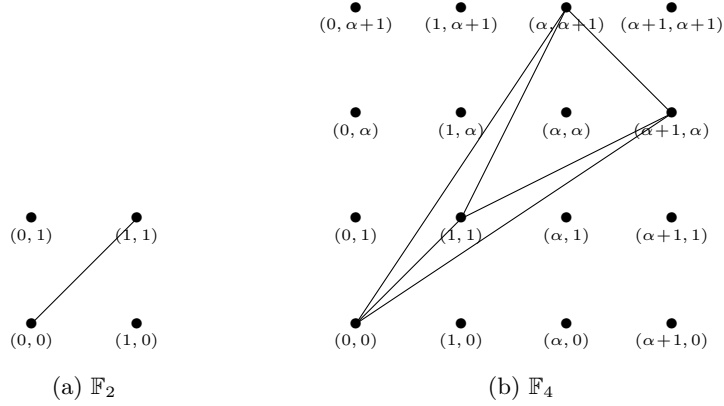

\subsubsection*{An anisotropic form}
An extreme case is the anisotropic form of dimension $2$ over $\F_2$ given by $[1,1]$.
Here it is clear that $\Gcal_{[1,1], 0}$ consists only of isolated points, while $\Gcal_{[1,1], 1}$ is the complete graph on 4 vertices.

\subsection{Finite fields}

Over finite fields we can give a complete classification of those representation graphs of quadratic forms which are connected. 
We will give a list of all exceptions in \Cref{thm:finite:field:Connectedness}.
In fact, we prove a stronger result by computing the diameter of these graphs in \Cref{thm:diameter:finite:fields}, with only one exception.
The exceptional case is that of hyperbolic planes, where we find an upper bound that exceeds the actual value by at most $1$.

\begin{theorem} \label{thm:diameter:finite:fields}
Let $F=\F_{p^m}$ be a finite field and $q$ be a quadratic form over $F$. Then the diameter of $\Gcal_{q,a}$ is as given in \Cref{tbl:diameter}.

\begin{table}[!ht]
    \renewcommand{\arraystretch}{1.25}
        \begin{tabular}{|c|c|c|c|c|}
            \hline
            cond. on $a$ & $\dim(q)$ & cond. on $q$ & cond. on $F = \F_{p^m}$ & $\diam(\Gcal_{q, a})$\\
            \hline \hline
            \multirow{2}{*}{$a = 0$} & \multirow{2}{*}{arbitrary} & \multicolumn{1}{c|}{isotropic} & \multirow{2}{*}{arbitrary} & \multicolumn{1}{c|}{2}\\ \cline{3-3} \cline{5-5}
            & & \multicolumn{1}{c|}{anisotropic} & & \multicolumn{1}{c|}{$\infty$}\\
            \hline
            
            \multirow{10}{*}{$a \neq 0$} & \multirow{2}{*}{$\dim(q) = 1$} & \multicolumn{1}{c|}{$a\det(q) = 1$} & \multicolumn{1}{c|}{$m = 1$} & \multicolumn{1}{c|}{$\frac{p - 1} 2$}\\
            \cline{3-5}
            & & \multicolumn{2}{|c|}{otherwise} & $\infty$ \\
            \cline{2-5}
            & \multirow{5}{*}{$\dim(q) = 2$} & \multirow{2}{*}{$q \cong \bH$} & $|F| \leq 4$ & $\infty$ 
            \\
            \cline{4-5}
            & & & $|F| \geq 5$ & 3 or 4
            \\
            \cline{3-5}
            & & \multirow{3}{*}{$q \not\cong \bH$} & $F = \F_2$ & 1
            \\
            \cline{4-5}
            & & & $F \in \{\F_3, \F_4\}$ & 2\\
            \cline{4-5}
            & & & otherwise & 3
            \\
            \cline{2-5}
            & \multirow{2}{*}{$\dim(q) = 3$} & $\det(q) = -a$ & \multirow{2}{*}{arbitrary} & 2
            \\
            \cline{3-3} \cline{5-5}
            & & $\det(q) \neq -a$ & & 3
            \\
            \cline{2-5}
            & $\dim(q) \geq 4$ & arbitrary & arbitrary & 2
            \\
            \hline
        \end{tabular}
    \caption{Diameter of representation graphs.}
    \label{tbl:diameter}
\end{table}
\end{theorem}

\begin{proof}
    The proof of the theorem is postponed to the following sections.
    Here we systematically collect where to find which part of the argument.

    First, the case $a = 0$ is treated in \Cref{prop:isotropic:diameter}, so from here on let $a \neq 0$ and we will proceed by dimension.
    Forms of dimension $1$ have already been considered in \Cref{sec:first:examples}.
    
    In the case of $2$-dimensional forms we distinguish between isotropic such forms, i.e. hyperbolic planes, and anisotropic such forms.
    For hyperbolic planes we will derive the upper bound $\diam(\Gcal_{\bH,a}) \leq 4$ in \Cref{prop:hyp:plane:connected} and the lower bound $\diam(\Gcal_{\bH,a}) \geq 3$ in \Cref{lem:Va+VaCapV0trivial}.
    For the anisotropic case we combine the explicit cases in \Cref{sec:first:examples} with the general bounds in \Cref{lem:DiamLowerBound} and \Cref{lem:diamLEQ3}.
    
    Forms of dimension $3$ are treated in \Cref{lem:diameter:dim3}, except for the case of the form $q = \bH\perp \qf{2}$, $a = 1$ over $F =\F_3$.
        But with \Cref{lem:diameter:dim3} and \Cref{lem:symmetryreduction} \ref{lem:symmetryreduction:3}, this case follows from the equality
        \[(1, 0, 0) = (2, 2, 0) + (1, 2, 1) + (1, 2, 2) \in V_1 + V_1 + V_1.\]
     Finally forms of dimension at least $4$ are considered in \Cref{lem:diameter:dim4}.    
\end{proof}

We used a computer algebra system to explicitly compute the diameter of representation graphs of the hyperbolic plane over finite fields with up to 2000 elements.
These computations suggest that the diameter of a representation graph of a hyperbolic plane is $4$ only for very small fields and is $3$ once the field has at least $11$ elements.
Unfortunately, we were unable to find any integer $f$ for which we could prove that $\diam(\Gcal_{\bH,a}) = 3$ for all finite fields with $|F| \geq f$.
So we offer only the following conjecture.

\begin{conjecture}
    Let $F$ be a finite field with $f$ elements. 
    For $a \in F^\ast$ we have 
    \[\diam(\Gcal_{\bH, a}) = \begin{cases}
        4, &\text{ if $f \leq 9$}\\
        3, &\text{ if $f \geq 11$}
    \end{cases}\]
\end{conjecture}

The explicit information about the diameter of the representation graphs of quadratic forms over finite fields in \Cref{thm:diameter:finite:fields} condenses into the following result about the connectedness of such graphs.

\begin{theorem}\label{thm:finite:field:Connectedness}
    Let $q$ be a non-degenerate quadratic form over a finite field $F$ of dimension at least 1.
    Then the graph $\Gcal_{q,a}$ is connected, except in the following cases
    \begin{enumerate}
        \item $a = 0$ and $q$ is anisotropic over $F$, or 
        \item $\dim(q) = 1$ but $a \not\in \D(q)$ or $F \not\cong \F_p$ for $p$ prime, or
        \item $a \in F^\ast$, $F \in \{ \F_2,\F_3,\F_4\}$ and $q = \bH$.
    \end{enumerate}
\end{theorem}

\subsection{Structural properties of the diameter}
As a preparation for the proofs of the above theorems, we collect some structural information about the diameter and connectedness of representation graphs, where possible over arbitrary fields.

The action of the isometry group of a quadratic form on the vertex set of a representation graph, which is just the underlying vector space of the quadratic form, allows to restrict adjacency-related concepts to prototypical situations.
In particular, it is sufficient to consider distances and paths of vertices to the origin.

\begin{lemma}\label{lem:symmetryreduction}
    Let $F$ be a field and $q$ be a quadratic form defined on the vector space $V$ and $a\in F$. Let $\iso(q)$ be the isometry group of $q$ acting on $V$.
    \begin{enumerate}
        \item \label{lem:symmetryreduction:1} The group $\iso(q)$ acts on the set of paths in $\Gcal_{q,a}$ in the natural way, preserving lengths, i.e. if $P$ is a $u$-$w$-path of length $\ell$ in $\Gcal_{q,a}$, then $\sigma(P)$ is a $\sigma(u)$-$\sigma(w)$-path of length $\ell$.
        \item \label{lem:symmetryreduction:2} Let $v_1, v_2, w_1, w_2\in V$ such that $w_1 - v_1 \neq 0 \neq w_2 - v_2$ and $q(w_1 - v_1) = q(w_2 - v_2)$.
        Then $\dist(v_1, w_1) = \dist(v_2, w_2)$ in $\Gcal_{q, a}$.
        \item \label{lem:symmetryreduction:3} For all $b\in \Dtilde(q)$, let $v_b\in V \setminus \{0\}$ such that $q(v_b) = b$. 
        We then have 
        \[\diam(\Gcal_{q, a}) = \sup\{\dist(0, v_b)\mid b\in \Dtilde(q)\}.\]
    \end{enumerate} 
\end{lemma}

\begin{proof}
    \ref{lem:symmetryreduction:1} follows from elementary considerations.
    Then \ref{lem:symmetryreduction:2} follows from \ref{lem:symmetryreduction:1}, since $\iso(q)$ acts transitively on 
    \[
        \left\{ (u,w) \in V \times V \mid q(u-w) = b \right\}
    \]
    for any $b \in F$.
    Finally for \ref{lem:symmetryreduction:3} we observe that by translation invariance, we have 
    \[\diam(\Gcal_{q, a}) = \sup\{\dist(0, v)\mid v\in V\}.\]
    By \ref{lem:symmetryreduction:2} we have $\dist(0, v) = \dist(0, w)$ if $q(v) = q(w)$ and $v \neq0 \neq w$, so we can take the supremum over $\{v_b\mid b\in \Dtilde(q)\}$.
\end{proof}
In the remainder we will often reduce the study of arbitrary $u$-$w$-paths to such paths starting at the origin without explicitly referring to this lemma.

Next, we will show that for fixed $a$ the diameter of $\Gcal_{q,a}$ of an isotropic quadratic form $q$ can be bounded by the diameter of $\Gcal_{\bH,a}$.
This follows from the fact that every vector is contained in a subform isometric to the hyperbolic plane.
While this is not difficult to see we include an argument for the convenience of the reader, as we are not aware of a direct reference. 

\begin{lemma}\label{lem:VectorsInHyperbolicSpaces}
    Let $q$ be an isotropic quadratic form defined on a vector space $V$ over a field $F$ and $v\in V$.
    Then there exists a $2$-dimensional subspace $W\subseteq V$ with $v\in W$ and $q|_W\cong\bH$.
\end{lemma}

\begin{proof}
    Since $q$ is isotropic, there is a subspace $W'$ with 
    \[q|_{W'} \cong \bH \subseteq q.\]
    Since $\bH$ is universal, we find some $w \in W'$ with $q(w) = q(v)$.
    By Witt's Extension Theorem \cite[Theorem 8.3]{MR2427530}, there is an isometry $\sigma \in \og(q)$ with $\sigma(w) = v$ and the claim follows by putting $W = \sigma(W')$.
\end{proof}

We can now apply this to obtain a bound on the diameter of representation graphs of isotropic quadratic forms.

\begin{corollary} \label{cor:BoundIsotropicDiameter}
    Let $q$ be an isotropic quadratic form and $a\in F$.
    We then have 
    \[\diam(\Gcal_{q, a}) \leq \diam(\Gcal_{\bH, a}),\]
    where we use the conventions $\infty \leq \infty$ and $n \leq \infty$ for all $n\in \N$.
\end{corollary}

\begin{proof}
    By translation invariance it suffices to consider $0$-$v$-paths in $\Gcal_{q,a}$.
    If $v$ is an arbitrary element it is contained in a subform $q' \cong \bH$ of $q$ which is isomorphic to the hyperbolic plane (cf. \Cref{lem:VectorsInHyperbolicSpaces}).
    Therefore we find a $0$-$v$-path in $\Gcal_{q,a}$ which is contained in the subgraph $\Gcal_{q',a} \cong \Gcal_{\bH,a}$. In any case the length of a minimal $0$-$v$-path in $\Gcal_{q,a}$ is bounded by the diameter of $\Gcal_{\bH,a}$.     
\end{proof}

On a side note, it is also easy to see that the connectedness of representation graphs is related to structural aspects of the underlying quadratic forms.
To be precise, if $q$ and $q'$ are quadratic forms over a field $F$ and $a \in F$, such that $\Gcal_{q,a}$ and $\Gcal_{q',a}$ are connected, then $\Gcal_{q\perp q',a}$ is also connected.

\begin{proposition} \label{lem:universal:extension}
    Let $q$ be a quadratic form from which we can split orthogonally a subform $q'$ which is universal and for which $\Gcal_{q',a}$ is connected.
    Then $\Gcal_{q,a}$ is also connected. 
\end{proposition}

\begin{proof}
    Write $q = q' \perp q''$, where $q'$ is the universal subform.
    Let $(v,w) \in V'\times V''$ be arbitrary. 
    Say $q''(w) = b$. 
    Then by the universality of $q'$ there exists $v'$ such that $q'(v') = a-b$.
    Then 
    \[
        (q' \perp q'')((v,w)-(v-v',0)) = q'(v') + q''(w) = a - b + b = a.
    \]
    So there is an edge between $(v,w)$ and $(v-v',0)$, and since $G_{q',a}$ is connected, there is a $(0,0)$-$(v,w)$-path in $\Gcal_{q,a}$.
\end{proof}

\subsection{Representation graphs for $a=0$}\label{subsec:rg:0}

The connectedness of the representation graph $\Gcal_{q,0}$ depends only on whether $q$ is isotropic or not.
Clearly, if $q$ is anisotropic, the representation graph $\Gcal_{q,0}$ consists only of isolated points. 
We will show that if $q$ is isotropic, the graph $\Gcal_{q,0}$ has a diameter of $2$ and is therefore always connected.

\begin{proposition} \label{prop:isotropic:diameter}
    Let $F$ be a field. 
    Let $q$ be an isotropic quadratic form over $F$. 
    Then $\diam(\Gcal_{q,0}) = 2$ so in particular $\Gcal_{q,0}$ is connected.
\end{proposition}

\begin{proof}
    First, by \Cref{cor:BoundIsotropicDiameter} we have $\diam(\Gcal_{q,0}) \leq \diam(\Gcal_{\bH,0})$. 
    So we only need to consider $\Gcal_{\bH,0}$ and it is enough to bound the length of a shortest $0$-$v$-path by translation invariance.
    
    For $\bH \cong [0,0]$ the standard basis $e_1,e_2$ consists of isotropic vectors.
    Let $v \in F^2$.
    If $v$ is isotropic, i.e. $\bH(v) = 0$, there is clearly a path of length at most $1$.
    Otherwise we write $v = \lambda_1 e_1 + \lambda_2 e_2$ with $\lambda_1 \neq 0 \neq \lambda_2$, so $(0,\lambda_1 e_1,\lambda_1 e_1 + \lambda_2 e_2 = v)$ is a $0$-$v$-path of length $2$.
\end{proof}

\subsection{Hyperbolic planes} \label{subsec:hyp:plane}
We collect some results specific to hyperbolic planes, not restricted to finite fields.

\begin{lemma} \label{lem:hyp:group:generators}
    Let $F$ be a field with at least $5$ elements. 
    The set $\{(x,x^{-1})\ \mid x \in F^\ast \}$ is a group generating set of $F^2$.
\end{lemma}

\begin{proof}
    We show that the set $\{(x, x^{-1})\ \mid x \in F^\ast\}$ generates all vectors of the form $(0,f)$ with $f \in F$. If $x,y,x+y \neq 0$ then
    \[
        (x,x^{-1}) + (y,y^{-1}) - (x+y,(x+y)^{-1}) = (0,x^{-1} + y^{-1} +(x+y)^{-1})
    \]
    represents every vector $(0,f)$ with 
    \[
        f \in S = \{ x^{-1} + y^{-1} - (x + y)^{-1} \ \mid x, y, x + y \neq 0 \}.
    \]

    We claim that $S \supseteq F^\ast$. 
    Note that $S$ is closed under multiplication with nonzero scalars, so it is sufficient to show that $S\cap F^\ast \neq \emptyset$.
    Note that for $x = 1$ and $y\notin\{0, -x\}=\{0, -1\}$ we have
    \[
        \frac{1}{x}+ \frac{1}{y} - \frac{1}{x + y} = \frac{y^2 + y + 1}{y+y^2},
    \]
    which is zero if and only if $y^2 + y + 1=0$.
    This quadratic equation has at most two solutions.
    So if $|F|\geq 5$, we find some nonzero element in $S$.
    
    Similarily, we get every vector $(f,0)$ and thus every element of $F^2$.
\end{proof}

With this explicit generating set we can derive an upper bound on the diameter in the hyperbolic case.

\begin{proposition} \label{prop:hyp:plane:connected}
    Let $\bH$ be the hyperbolic plane over a field $F$ and let $a \in F^\ast$. 
    \begin{enumerate}
        \item The graph $\Gcal_{\bH,a}$ is connected if and only if $|F| \geq 5$.
        \item If $|F| \geq 5$, then $\dist(0,v) \leq 3$ for all isotropic $v$ and $\diam(\Gcal_{q,a}) \leq 4$.
    \end{enumerate}
\end{proposition}

\begin{proof}
    We may assume $a = 1$ by the roundness of the hyperbolic plane, i.e. $q \cong a q$ for all $a \in F^\ast$ (see \Cref{sec:prelim}).
    If $|F| \leq 4$, then $\Gcal_{q,a}$ is disconnected as discussed in the examples in \Cref{sec:first:examples}.
    
    Now let $|F| \geq 5$.
    Since $\bH(x,y)=xy$, we have
    \[
        V_1 = \{ (x,x^{-1}) \ \mid x \in F^\ast \}
    \]
    which is a group generating set of $F^2$ by \Cref{lem:hyp:group:generators}, so the graph is connected by \Cref{lem:cayley:connected}. 

    For the bound on the diameter we observe that the proof of \Cref{lem:hyp:group:generators} shows that there is a path of length at most $3$ from the origin to $(b - 1, 0)$. 
    Concatenating this path with the edge $(1, 1)$ yields a path of length at most 4 from the origin to $(b, 1)$.    
    The assertion now follows from choosing $v_b = (b,1)$ in \Cref{lem:symmetryreduction}.    
\end{proof}

On the other hand, a simple argument gives a lower bound for the diameter in the hyperbolic case over arbitrary fields.

\begin{lemma}\label{lem:Va+VaCapV0trivial}
    Let $q$ be a quadratic form of dimension 2.
    For $a \in F^\ast$, we have
    \[(V_a + V_a) \cap V_0 = \{0\}.\]
    In particular, $\diam(\Gcal_{\bH,a}) \geq 3$.
\end{lemma}
\begin{proof}
    Let $v, w \in V_a$ be such that \[0 = q(v + w) = q(v) + q(w) + b_q(v, w),\]
    which implies $b_q(v, w) = -2a$.
    If $v, w$ were a basis, $q$ would be degenerate.
    So $v, w$ are linearly dependent and we have $v + w = \lambda v$ for some $\lambda \in F$.
    Finally we get
    \[0 = q(v + w) = q(\lambda v) = \lambda^2 q(v) = \lambda^2 a,\]
    so $\lambda = 0$ and $v + w = 0$.
\end{proof}

\subsection{On paths of length two} \label{sec:counting:lemma}
In the case of finite fields, it will often be helpful to construct paths between points, or to show that certain paths cannot exist, by purely enumerative considerations.
The following is our main tool for this.

\begin{lemma} \label{lem:NumberOfSums}
    Let $F$ be a field and $q$ be a quadratic form over $F$ with $\dim(q) = 2$. 
    For $a, b\in F^\ast$, we consider a vector $w \in V_a + V_b \setminus \{0\}$.
    \begin{enumerate}
        \item \label{lem:NumberOfSums1} There are at most two different tuples $(u, v) \in V_a \times V_b$ such that $w = u + v$.
        \item \label{lem:NumberOfSums2} Let $q(w) \neq 0$. 
        There exist two different tuples $(u, v) \in V_a \times V_b$ with $w = u + v$ if and only if there exists such a tuple with $u, v$ linearly independent.
    \end{enumerate}
    Now let $F$ be finite with $|F| = f \in \N$.
    We set \[k = \begin{cases}
            f + 1, &\text{ if $q$ is anisotropic}\\
            f - 1, &\text{ if $q$ is isotropic}
        \end{cases}\]
    \begin{enumerate}[resume]
        \item \label{lem:NumberOfSums3} If $w$ is such that there is a unique tuple $(u, v)\in V_a\times V_b$ with $u + v = w$, then $ab\in F^{\ast2}$.
        Conversely, if $ab \in F^{\ast2}$, then the number of $w$ such that there exists a unique tuple $(u, v) \in V_{a}\times V_b$ with $u + v = w$ is as given in \Cref{tbl:Unique:Rep}.
        \begin{table}[ht]
        \renewcommand{\arraystretch}{1.25}
        \begin{tabular}{|c||c|c|}
            \hline
            conditions on $a$ and $b$ & $\cha(F) \neq 2$ & $\cha(F) = 2$\\  
            \hline\hline
            $a = b$ & $k$ & $0$  \\ \hline
            $a \neq b$ & $2k$ & $k$  \\ \hline 
        \end{tabular}
        \caption{Number of vectors in $V_a + V_b$ with a unique decomposition} \label{tbl:Unique:Rep}
        \end{table}
        \item \label{lem:NumberOfSums4} The cardinality of $V_a + V_b$ is as in \Cref{tbl:SizeVa+Vb}.
    \begin{table}[ht]
    \renewcommand{\arraystretch}{1.25}
        \begin{tabular}{|c||c|c|}
            \hline
            conditions on $a$ and $b$ & $\cha(F) \neq 2$ & $\cha(F) = 2$\\  
            \hline\hline
            $ab \notin F^{\ast2}$ & $\frac{k^2}2$ & $0$ \\ \hline
            $a = b$ & $\frac{k^2}2 + 1$ & $\frac{k \cdot (k - 1)}2 + 1$ \\ 
            \hline 
            $ab \in F^{\ast2}, a \neq b$ & $\frac{k \cdot (k+2)}2$ & $\frac{k \cdot (k+1)}2$ \\
            \hline 
        \end{tabular}
        \caption{Size of $V_a + V_b$} \label{tbl:SizeVa+Vb}
        \end{table}
    \end{enumerate}
\end{lemma}
\begin{proof}
    \ref{lem:NumberOfSums1}: To count the number of tuples $(u, v) \in V_a\times V_b$ with $u + v = w$, it suffices to count the number of possible choices for $u \in V_a$, since $v = w - u$.
    Such $u$ must satisfy
    \begin{align}\label{eq:restriction}
        q(u) = a \quad \text{and} \quad b_q(u,w) = q(w) + a - b.
    \end{align}
    Since these are polynomial constraints of degrees $2$ and $1$ respectively, B\'ezout's theorem (see e.g. \cite[Theorem 3.1]{Schmid} for a proof in the affine case) gives that there are at most two solutions to this system of equations.

    \ref{lem:NumberOfSums2}: Let now $w$ be an anisotropic vector and $w = u + v$ with $u\in V_a, v \in V_b$ and $u, v$ linearly independent.
    We have $\dim(w^\perp) = 1$ and choose a vector $p \in w^\perp\setminus\{0\}$.
    Then $p$ is also anisotropic: If $\cha(F) =2$, this follows because $w^\perp = \spn(w)$ and $w$ is anisotropic by assumption. If $\cha(F) \neq 2$, then $w, p$ form an orthogonal basis of $V$, where $q(p) = 0$ would contradict the non-degeneracy of $q$.
    
    Since $p$ is anisotropic, we can consider the reflection at $p^\perp = \spn(w)$ given by
    \[\sigma:V\to V, x\mapsto x - \frac{b_q(p, x)}{q(p)}\cdot p.\]
    Since we have $\sigma(x) = x$ if and only if $x \in w^\perp$, we have $\sigma(u) \neq u$ and $\sigma(v) \neq v$ as otherwise we would obtain a contradiction to their linear independence.
    Now $\sigma$ is an isometry and we have
    \[w = \sigma(w) = \sigma(u + v) = \sigma(u) + \sigma(v),\]
    which yields a second tuple $(\sigma(u),\sigma(v))$ in $V_a\times V_b$ that adds up to $w$.

    Conversely let $v = \lambda u$ with $\lambda \in F^\ast$.
    Then
    \begin{align}\label{eq:blambda^2a}
        b = q(v) = q(\lambda v) = \lambda^2 q(v) = \lambda^2 a.
    \end{align}
    We claim that if $(x, y) \in V_a \times V_b$ is such that $x + y = w$ and $(x, y) \neq (u, v)$, then $x, y$ must be linearly independent.
    To see this, we assume on the contrary that $x,y$ are linearly dependent.
    It is easy to see that $x,y \in \spn(u)$.  
    Now $x \in V_a \cap \spn(u) = \{ \pm u \}$ and $y \in V_b \cap \spn(v) = \{ \pm v \} = \{ \pm \lambda u \}$.
    This makes it easy to see that $x = u$ and $y=v$ are actually required.
       
    So let $x \in V_a,y \in V_b$ be linearly independent.
    Then 
    \[q(u) + q(v) + b_q(u, v) = q(w) = q(x) + q(y) + b_q(x,y),\]
    and so 
    \begin{equation*}
       b_q(x, y) = b_q(u,v) = q(u+v) - q(u) -q(v) = 2\lambda a.
    \end{equation*}
   This implies that the Gram matrix for $b_q$ in the basis $(x, y)$ is singular, which is clearly a contradiction.

    \ref{lem:NumberOfSums3}: The first part follows directly from \eqref{eq:blambda^2a}.
    Now let $ab \in F^{\ast2}$.
    If $a = b$, we have $w = 2u$ for $u \in V_a$. 
    If $2 = 0$, this cannot happen for $w \neq 0$ and otherwise there are $k$ such $u$ by \Cref{prop:NumberOfRepresentations}.
    
    If $a \neq b$ there is $\lambda \in F^\ast \setminus \{\pm 1 \}$ such that $b = a\lambda^2$.
    Then $w = (1 \pm \lambda) u$ for $u \in V_a$, and there are $k$ such $u$ by \Cref{prop:NumberOfRepresentations}. 
    The claim now follows easily.
    
    \ref{lem:NumberOfSums4}: 
    Let first $\cha(F) \neq 2$.
    If $ab \notin F^{\ast2}$, every vector $w = u + v$ with $u \in V_a, v \in V_b$ has exactly two such representations as a sum.
    If $a = b$, we have $k$ tuples $(u,u)$, $k$ tuples $(u,-u)$, and $k(k-2)$ tuples $(u,v)$ with $u,v$ linearly independent, giving $k + 1 + \frac{k(k-2)}{2}$ different $w \in V_a + V_a$.
    Similarly, if $ab \in F^{\ast2}$ but $a\neq b$, we obtain $2k + 0 + \frac{k\cdot(k-2)}2 = \frac{k\cdot{k+2}}2$ elements in $V_a + V_b$.

    Note that if $\cha(F) = 2$, $F$ is perfect and every element in $F$ is a square, which eliminates the case $ab \not\in F^{\ast2}$. 
    The rest of the calculation is similar.
\end{proof}

The version of B\'ezout's theorem used in the proof of \Cref{lem:NumberOfSums} \ref{lem:NumberOfSums1} can be shown elementarily:
Choosing coordinates, the linear constraint in \eqref{eq:restriction} can be used to write the second coordinate as a linear expression in the first one. 
Using this expression for the second coordinate in the quadratic constraint in \eqref{eq:restriction} yields a quadratic expression for the first coordinate, which has at most two solutions.\medskip

In the following result we explicitly determine which values $a \in F^\ast$ occur as the value of $q$ at an endpoint of paths of length two starting at the origin.
This proved to be a useful criterion for explicit computations.

\begin{proposition}
    Let $F$ be a field and let $q$ be a quadratic form of dimension 2 representing 1.
    For $a \in F^\ast$, there exists some $w \in V_1 + V_1$ with $q(w) = a$ if and only if
    \begin{equation*}
        \det(q) \cdot (4a - a^2) \text{ is a square} \quad \text{if $\cha(F) \neq 2$}
    \end{equation*}
     and 
    \begin{equation*}
        \Delta(q) + \frac{1}{a} \in \wp(F) \quad \text{if $\cha(F) = 2$.}
    \end{equation*}    
\end{proposition}
\begin{proof}
    Let $u, v \in V_1$ such that $w = u + v$ with $q(w) = a$.
    If $v \in \{u, -u\}$, the assertion is clear.
    Otherwise $u, v$ are linearly independent and thus form a basis and we have $b_q(u, v) = a - 2$.
    
    If $\cha(F) \neq 2$, $q$ is isometric to the quadratic form $x \mapsto x^T A x$ with
    \begin{equation} \label{eq:exp:q:1}
        A = \begin{pmatrix} 1 & \tfrac{a-2}{2} \\ \tfrac{a-2}{2} & 1 \end{pmatrix} \quad \text{with }\det(q) = \det(A) = \frac{1}{4} \cdot (4a - a^2),
    \end{equation}
    so $\det(q) \cdot (4a - a^2)$ is a square.
    
    Now let $\cha(F) = 2$.
    Recall that $a - 2 = a \neq 0$, so $b_q(u, \frac1{a}\cdot v) = 1$ and we have
    \begin{equation}\label{eq:exp:q:2}
        q \cong\left[q(u), q(\tfrac1{a}\cdot v)\right] = \left[1, \tfrac{1}{a^2}\right],
    \end{equation}
    which has Arf invariant $\frac{1}{a^2} = \frac1a \in F/\wp(F)$.

    For the reverse direction, the case $a = 4$ is clear. 
    If $a \neq 4$, we can reverse the arguments: 
    The determinant / Arf invariant of $q$ is determined by the given condition and since we assume that $q$ to represents $1$, the isometry type of $q$ is determined uniquely as $\qf{1, \det(q)}$ / $[1, \Delta(q)]$.
    Then, $q$ is isometric to one of the forms \eqref{eq:exp:q:1} / \eqref{eq:exp:q:2} and a vector $w$ as desired can be constructed easily.
\end{proof}

\subsection{Finishing the proof} \label{sec:proof}
We start with a simple consequence of \Cref{lem:NumberOfSums}.

\begin{lemma} \label{lem:DiamLowerBound}
    Let $F$ be a finite field with $|F| \geq 7$ and let $q$ be an anisotropic quadratic form of dimension 2 over $F$ and let $a \in F^\ast$.
    Then $\diam(\Gcal_{q, a}) \geq 3$.
\end{lemma}

\begin{proof}
    Let $f = |F|$. 
    We have $|V_a| = f+1$ and by \Cref{lem:NumberOfSums} $|V_a + V_a| \leq \frac{(f+1)^2}{2}+1$.
    The union of these sets is precisely the set of vertices with distance at most $2$ to the origin, but it contains at most $\frac{f^2 + 4f + 5}{2}$ elements, which is strictly less than $|V| = f^2$ if $f \geq 7$.
\end{proof}

It is clear that over finite fields every anisotropic form is of dimension $1$ or $2$.
For $1$-dimensional forms the discussion in \Cref{sec:first:examples} already gives complete information about both the connectedness and the diameter of the associated representation graphs.
We will therefore discuss the case of $2$-dimensional anisotropic forms.

\begin{lemma}\label{lem:diamLEQ3}
    Let $F$ be a finite field and let $q$ be an anisotropic quadratic form of dimension $2$ over $F$ and let $a \in F^\ast$.
    Then $\diam(\Gcal_{q, a}) \leq 3$. 
\end{lemma}
\begin{proof}
    We may assume $a = 1$.
    Let $f = |F|$. 
    Let $v$ be such that $\dist(0,v) \geq 3$ and write $x = q(v) \in F\setminus\{0,1\}$.
    Since $\og(q)$ acts transitively on $V_x$ by \Cref{cor:OGtransitive}, it is enough to show that one vector $v'$ with $q(v') = x$ can be reached via a path of length 3 from the origin.    
    The existence of such a $v'$ follows from the observation that the set of points adjacent to points in $V_x$ intersects $V_1 + V_1$:
    By \Cref{lem:NumberOfSums} \ref{lem:NumberOfSums4}, we have $|V_x + V_1| \geq \frac{(f+1)^2}{2}$ and $|V_1 + V_1| \geq \frac{(f+1)\cdot f}2 +1$.
    Both sets are subsets of $V$ and $|V| = f^2$, so we have $(V_x + V_1) \cap (V_1 + V_1) \neq \emptyset$.
\end{proof}

From now on, we will restrict our attention to isotropic forms.
Quadratic forms over finite fields of dimension at least 3 are isotropic and by \Cref{cor:BoundIsotropicDiameter} the diameter of an associated representation graph is bounded by that of the graph of the corresponding hyperbolic plane.
It will now turn out that this can be strengthened significantly.

\begin{lemma} \label{lem:diameter:dim3}
    Let $q$ be a quadratic form of dimension $3$ over a finite field $F$ and $a \in F^\ast$.
    For $v \in V$ we have $\dist_{\Gcal_{q, a}}(0, v) > 2$ if and only if $v$ is isotropic and $\det(q) \neq -a$.
    In particular, if $|F| \geq 5$, we have
    \[\diam(\Gcal_{q, a}) = \begin{cases}
        2, &\text{ if $\det(q) = -a$}\\
        3, &\text{ else.}
    \end{cases}\]
\end{lemma}

\begin{proof}
    Let first $v$ be an anisotropic vector, i.e. $b := q(v) \neq 0$.
    Since $\dim(q) =3$ we necessarily have $\cha(F) \neq 2$, so $q|_{\spn(v)}$ is non-degenerate and therefore is a subform of $q$.
    This means that we can extend $v$ to a suitable basis of $V$ in which $q$ is isometric to $\qf b \perp \bH$. 
    Since $\Dtilde(\bH) = F$, there exists some vector $w \in v^\perp $ with $q(w) = a - \frac b4 = q(-w)$.
    Then $\tfrac{1}{2}v \pm w \in V_a$ and 
    \[v = \left(\frac12 v + w\right) + \left(\frac12 v - w\right) \in V_a + V_a.\]
    
    Let now $v \neq 0$ be an isotropic vector.
    Similar to above, we extend $v$ to a basis $(v, w, u)$ of $V$ such that in this basis, $q$ is isometric to $\bH \perp \qf c$ for some $c \in F^\ast$ and such that $q(v) = 0 = q(w)$ and $b_q(v, w) = 1$ (i.e. $v, w$ is a \emph{hyperbolic pair}).
    If $c \in aF^{\ast2}$, we may assume that $c = a$. 
    In this case we can write $v = (v + u) - u \in V_a + V_a$.
    
    Assume $c \not\in aF^{\ast2}$ and $v = x + y$ with $x,y \in V_a$. 
    Since $v = v + 0w + 0u$ we have
    \[ x = \alpha v + \lambda w + \mu u \quad \text{and} \quad y =\beta v - \lambda w - \mu u \]
    with $\alpha, \beta, \lambda, \mu \in F$ and $\alpha + \beta =1$.
    Now $\lambda \neq 0$, since $q(\alpha v + \mu u) = \mu^2 c \not\in aF^{\ast2}$ by assumption.
    But
    \[ \alpha\lambda + \mu^2c = q(x)= a = q(y) = -\beta\lambda + \mu^2 c\]
    implies $0 = \lambda\cdot(\alpha + \beta) = \lambda$, a contradiction.

    If $|F| \geq 5$ we know $\dist(0,v) \leq 3$ for isotropic vectors by \Cref{prop:hyp:plane:connected}.
    Together with the above, this proves the additional claim.
\end{proof}

Having dealt with forms of dimension at most $3$, we are left to consider forms of dimension at least $4$.

\begin{lemma} \label{lem:diameter:dim4}
    Let $q$ be a quadratic form of dimension at least 4 over a finite field and $a \in \D(q)$. 
    Then $\diam(\Gcal_{q, a}) = 2$.
\end{lemma}
\begin{proof}
    Since $V_a \subsetneq V$, we have $\diam(\Gcal_{q, a}) \geq 2$.
    Now let $v \in V$ with $q(v) = x \neq a$.
    By \Cref{thm:ClassificationQFFIniteFields}, $q$ is isotropic, so there exists a subspace $W$ with $v \in W \subseteq V$ of dimension $2$ such that $q|_W\cong\bH$ by \Cref{lem:VectorsInHyperbolicSpaces}.
    As in the proof of \Cref{prop:isotropic:diameter}, there are isotropic vectors $w_1, w_2 \in W$ such that $v = w_1 + w_2$.
    
    Now $\dim(W^\perp) \geq 2$ and forms of dimension at least 2 over finite fields are universal, so there is some $u \in W^\perp$ such that $q(u) = a = q(-u)$.
    Then
    \[v = w_1 + w_2 = (w_1 + u) + (w_2 - u) \in V_a + V_a\]
    proves the lemma.
\end{proof}

\begin{remark}
    In the proofs of \Cref{lem:diameter:dim3} and \Cref{lem:diameter:dim4}, we relied on the classification of quadratic forms over finite fields. 
    In fact, we only needed the property that over finite fields, every form of dimension at least 3 is isotropic, or equivalently, that every form of dimension $2$ is universal (we used both characterizations in the proof).
    In quadratic form theory, this behavior is expressed using the \emph{$u$-invariant} of a field $F$, which in its most traditional form is defined as
    \[u(F)=\sup\{\dim(q)\mid q\text{ is anisotropic}\} \in \N \cup \{\infty\}.\]
    According to the above, the statements of \Cref{lem:diameter:dim3} and \Cref{lem:diameter:dim4} apply to all fields $F$ with $|F| \geq 5$ and $u(F) = 2$.
\end{remark}

\section{Girth} \label{sec:girth}

\subsection{Determining the Girth}

We will now determine the \emph{girth} i.e. the minimum length of a cycle in the representation graphs $\Gcal_{q, a}$ for non-degenerate quadratic forms over finite fields.
We will see that, with few exceptions, this graph invariant is almost always given by 3.
This motivates us in particular to characterize triangles, and we can do this in an algebraic way.
\begin{lemma} \label{lem:triangle}
    Let $q$ be a quadratic form over $F$ and $a \in \Dtilde(q)$. 
    Let $v,w \in V_a \setminus \{0\}$ distinct and $W = \spn(v,w)$.
    If $a=0$, then $0,v,w$ form a triangle in $\Gcal_{q,0}$ if and only if $W$ is totally isotropic.
    
    If $a\neq 0$, then there exists a triangle in $\Gcal_{q,a}$ with vertices in $W$ if and only if
    \begin{align*}
        q|_{W} &\cong [a, a^{-1}]\\
        \text{or} \quad q|_{W} &\cong \qf{a} \text{ and } \cha(F) = 3.
    \end{align*}
\end{lemma}

\begin{proof}
    The vectors $0, x, y$ form a triangle if and only if $a = q(x - y) = q(x) + q(y) - b_q(x, y)$, i.e. if and only if $b_q(x, y) = a$.
    With this information, if $0, x, y$ form a triangle in $W$, we can compute the isometry type of $q|_W$, depending on $\cha(F)$, using a subset $x, y$ as a basis of $W$.
    
    Conversely, if one of the conditions on the isometry type of $q$ and the characteristic of $F$ is satisfied, we can explicitly construct a triangle with vertices in $W$ as follows:    
    If $a=0$, we can choose $0,v,w$ and if $a \neq 0$, we can choose $0,x,ay$ and $0,x,-x$ respectively, as the vertices of the triangle.
\end{proof}

Note that if $\cha(F) \neq 2$ the forms $[a, a^{-1}]$ and $\qf{a,3a}$ are isometric, in the language of the latter form, for the basis $x,y$, we could also choose a triangle with vertices $0,x,\tfrac{1}{2}(x+y)$.

\begin{theorem} \label{thm:girth}
    Let $F$ be a finite field of cardinality $|F| = p^m$ with $p$ a prime, $m \in \N$ and $a\in\Dtilde(q)$, i.e. $\Gcal_{q,a}$ does not consist of isolated points. 
    Then the girth of $\Gcal_{q,a}$ is as given in \Cref{tbl:girth}.  
        \begin{table}[!ht]
            \renewcommand{\arraystretch}{1.25}
            \begin{tabular}{|c|c|c|c|c|}
                \hline
                cond. on $a$ & $\dim(q)$ & cond. on $F$ & cond. on $q$ & $\g(\Gcal_{q, a})$\\
                \hline \hline
                \multirow{2}{*}{$a = 0$} & \multirow{2}{*}{arbitrary} & \multicolumn{1}{c|}{$F = \F_2$} & \multicolumn{1}{c|}{$q \in \{\bH, \bH \perp [1, 1]\}$} & \multicolumn{1}{c|}{4}\\ \cline{3-5}
                & & \multicolumn{2}{c|}{otherwise} & \multicolumn{1}{c|}{$3$}\\
                \hline
                
                \multirow{6}{*}{$a \neq 0$} & $\dim(q) = 1$ & $\cha(F) \neq 2$ & none & $p$\\
                \cline{2-5}
                & \multirow{4}{*}{$\dim(q) = 2$} & $F = \F_2$ & $q \cong \bH$ & $\infty$ 
                \\
                \cline{3-5}
                & & $F = \F_3$ & none & 3\\
                \cline{3-5}
                & & none & $q \cong [a, a^{-1}]$ & 3\\
                \cline{3-5}
                & & \multicolumn{2}{c|}{otherwise} & 4
                \\
                \cline{2-5}
                & $\dim(q) \geq 3$ & none & none & 3
                \\
                \hline
            \end{tabular}
        \caption{Girth of representation graphs.}
        \label{tbl:girth}
    \end{table}
\end{theorem}
\begin{proof}
    For the case $a=0$ we observe that $q$ is isotropic, as $a \in \Dtilde(q)$.
    If $v$ is an isotropic vector, then $(0, v, \lambda v, 0)$ is a triangle if $\lambda \not\in\{ 0,1\}$.
    So only the case $F = \F_2$ remains. 
    For the forms $\bH$ and $\bH \perp [1,1]$ the girth can be computed explicitly.
    In all other cases $\iw(q) \geq 2$ and $V$ contains a totally isotropic subspace of dimension 2.
    Therefore we find linearly independent $v, w \in V_0$, such that $v - w \in V_0$, so $(0, v, w, 0)$ is a cycle of length 3.
    
    Now let $a \neq 0$.
    It immediately follows that $\g(q) =p$ if $\dim(q) = 1$ and $\cha(F) = p \neq 2$ and $\g(\bH) = \infty$ for $\bH$ over $\F_2$. 
    In general we observe that $\g(\Gcal_{q,a}) \leq 4$ in all remaining cases.
    For this we choose $v,w \in V_a$ linearly independent, which we can do as $|V_a| \geq 3$ (see \Cref{prop:NumberOfRepresentations}).
    Then $(0, v, v + w, w, 0)$ is a cycle of length $4$.

    This reduces the problem in the remaining cases to determining when $\g(\Gcal_{q,a}) = 3$, that is, when $\Gcal_{q,a}$ contains a triangle.      
    For forms of dimension $2$ we can use \Cref{lem:triangle}.
    A form of dimension at least $3$ contains the form $[a, a^{-1}]$ as a subform, so with \Cref{lem:triangle} we find a triangle in the corresponding representation graph.
    To see that this form is indeed contained as a subform, we recall from the classification of non-degenerate quadratic forms over finite fields (see \Cref{thm:ClassificationQFFIniteFields}) that the isometry class of $q$ is uniquely determined by $\dim(q)$ and $\det(q)$ if $\cha(F) \neq 2$ and by $\dim(q)$ and $\Delta(q)$ if $\cha(F) = 2$.
    Then 
    \begin{align*}
        q &\cong [a, a^{-1}] \perp \qf{\tfrac{1}{3}\det(q)} \perp (\dim(q)-3) \times \qf{1} \\ 
        \text{or} \quad q &\cong [a, a^{-1}]\perp[1,\Delta(q)] \perp \tfrac{\dim(q)-4}{2}\times [1,1]
    \end{align*}
    and $q$ contains at least one of the above $2$-dimensional subforms as claimed.
\end{proof}

Note that for any graph $G$, we have $g(G) = 3$ if and only if $\omega(G) \geq 3$, where $\omega(G)$ denotes the \emph{clique number} of $G$, i.e. the size of a clique in $G$ of maximal cardinality.
Therefore, partial information on the girth could have been extracted from the computation of clique numbers for representation graphs of quadratic forms carried out in \cite[Corollary 3.7, Theorem 4.4, Theorem 4.6]{lorenz:2023:cir}.

\subsection{Cycles of Length 3} 
We complement the computation of the minimum length of a cycle in a representation graph by counting the number of such cycles, starting with cycles of length $3$.
Again linearizing, we first consider those cycles that include the origin, i.e. cycles of the form $(0, v, w, 0)$.
The vector space spanned by the vertices of such a cycle has dimension $1$ or $2$, and we will count them according to this distinction.

For fixed $q$, $a \in \Dtilde(q)$ and $k \in\{1, 2\}$, we will write
\[
    c_k = \text{number of cycles $(0, v, w, 0)$ in $\Gcal_{q, a}$ with $\dim(\spn(v, w)) = k$.}
\]
While $c_k$ clearly depends on $q$ and $a$, we assume that this is clear from the context.
Then the number of all cycles of length $3$ in $\Gcal_{q,a}$ can be related to $c_1,c_2$ as follows.

\begin{proposition} \label{prop:CyclesLength3}
    Let $F$ be a finite field with $f$ elements, $q$ be a quadratic form over $F$, and let $a \in \Dtilde(q)$.
    The number of all cycles of length 3 in $\Gcal_{q, a}$ is given by 
    \[
        (c_1 + c_2)\cdot \frac{|V|}{3}.
    \]
    where $c_1$ and $c_2$ are as in \Cref{lem:CyclesLength3a=0} and \Cref{lem:CyclesLength3With0} depending on $a$ and $q$.
\end{proposition}

\begin{proof}
    Counting in two ways, we obtain
    \begin{align*}
        3 \cdot \# \text{ cycles $C$ of length 3} 
        &= |\{(C, v)\mid C \text{ cycle of length 3, $C$ contains $v$}\}|\\
        &= |V| \cdot \# \text{ cycles $C$ containing $v$}\\
        &= |V| \cdot \# \text{ cycles $C$ containing $0$}.
    \end{align*}
    and the claim follows.
\end{proof}

Now we are left to determine $c_1$ and $c_2$, we do this first in the case of adjacency induced by isotropic vectors, i.e. $a = 0$.

\begin{lemma} \label{lem:CyclesLength3a=0}
    Let $q$ be an isotropic quadratic form over a finite field $F$ with $|F| = f$.
    Let $t_i$ be the number of totally isotropic subspaces of dimension $i$.
    For the graph $\Gcal_{q, 0}$, we have
    \begin{equation*}
        c_1 = t_1 \cdot \binom{f - 1}{2} \quad \text{and} \quad c_2 = t_2 \cdot \frac{(f^2 - 1) \cdot (f^2 - f)}{2}.
    \end{equation*}
\end{lemma}
\begin{proof}
    Each cycle $(0, v, w, 0)$ with $v, w$ linearly dependent uniquely determines a totally isotropic subspace of dimension 1 and for each such subspace, there are clearly $\binom{f - 1}2$ choices of $v, w$ to construct a cycle.

    If $C = (0, v, w, 0)$ is a cycle of length 3 with $v, w$ linearly independent, then $W = \spn(v, w)$ is a totally isotropic subspace (see also \Cref{lem:triangle}).    
    For any totally isotropic subspace $W$ of dimension $2$, vectors $x,y \in W$ correspond to such a cycle if and only if $x,y$ are a basis of $W$, giving the number in the assertion.
\end{proof}

To obtain explicit numbers one can use the well-known number of totally isotropic subspaces of fixed dimension, see for example \cite[Theorem 1.41 (i)]{HirschfeldThas:GaloisGeometries} for a modern exposition of this question, albeit in a geometric language.

Now we turn to the case $a \neq 0$.
Here we can base the calculation of the number of cycles of length 3 on suitable decompositions of $q$ into lower dimensional subforms.
Over fields of characteristic not 2, we use subforms of the kind $\qf a$ (non-degenerate only if $\cha(F) \neq 2$), and over fields of characteristic not 3, we use subforms of the kind $[a,a^{-1}]$ (non-degenerate only if $\cha(F) \neq 3$).
Obviously, the respective scopes of these approaches have a huge overlap, but we know of no characteristic-free approach.

\begin{lemma} \label{lem:CyclesLength3With0}
    Let $F$ be a finite field, $a \in \D(q)$.
    \begin{enumerate}
        \item \label{lem:CyclesLength3With00} We have $c_1 = 0$ if $\cha(F) \neq 3$ and $c_1 = \tfrac{|V_a|}{2}$ if $\cha(F) = 3$.
        \item \label{lem:CyclesLength3With01} Let $\cha(F) \neq 2$. If there exists a form $q'$ such that $q = \qf{a} \perp q'$, then
        \[c_2 = \frac12\cdot |V_{q, a}| \cdot |V_{q', 3a}|,\]
        otherwise $c_2 = 0$.
        \item \label{lem:CyclesLength3With02} Let $\cha(F) \neq 3$. If there exists a form $q'$ such that $q = [a, a^{-1}] \perp q'$, then
        \[c_2 = \frac{|V_{[a, a^{-1}], 3a}| \cdot |\og(q)|}{|\og([a, a^{-1}])| \cdot |\og(q')|},\]
        otherwise $c_2 = 0$.
    \end{enumerate} 
\end{lemma}
\begin{proof}
    \ref{lem:CyclesLength3With00}: 
    By \Cref{lem:triangle}, we see that $c_1 = 0$ when $\cha(F) \neq 3$ and when $\cha(F) =3$, for any cycle $(0, v, w, 0)$ with $v, w$ linearly dependent, we have $w = -v$ and vice versa.

    \ref{lem:CyclesLength3With01}:
    Let $v \in V_a$.
    First we count the number of ways to extend $0, v$ to a cycle $(0, v, w, 0)$ of length $3$.
    Since $q(v) = a \neq 0$, $v$ splits off orthogonally.
    We write $w = \lambda v + w_0$ with $w_0 \perp v$ and $\lambda \in F$.
    By the proof of \Cref{lem:triangle}, we have 
    \[a = b_q(v, w) = b_q(v, \lambda v + w_0) = b_q(v, \lambda v) + b_q(v, w_0) = \lambda b_q(v, v) = 2a\lambda,\]
    so $\lambda = \frac12$.
    This implies
    \[a = q(w) = q\left(\frac12 v + w_0\right) = \frac14 q(v) + q(w_0) +\frac12 b_q(v, w_0)= \frac a4 + q(w_0),\]
    so $q(w_0) = \frac34 a$.
    The number of possibilities to extend $0, v$ to a cycle of length $3$ is thus given by $|V_{q', 3a/4}| = |V_{q', 3a}|$.
    This counts every cycle of length 3 exactly twice.

    \ref{lem:CyclesLength3With02}:
    Each cycle of length 3 of the form $(0, v, w, 0)$ with $v, w$ linearly independent determines a unique subspace $W = \spn(v, w)$ with $q|_W \cong [a, a^{-1}]$ by \Cref{lem:triangle}.
    We first determine the number of such subspaces and then count the number of cycles within such a subspace.
    
    By Witt's Extension Theorem \cite[Theorem 8.3]{MR2427530}, $\og(q)$ acts transitively on the set of subspaces $W$ such that $q|_W \cong [a, a^{-1}]$.
    Since isometries preserve orthogonality, the stabilizer of a fixed subspace $W$ such that $q|_W \cong [a, a^{-1}]$ is given by $\og([a, a^{-1}]) \times \og(q')$.
    Therefore, by the orbit stabilizer theorem, we deduce that there are 
    \[\frac{|\og(q)|}{|\og([a, a^{-1}]) \times \og(q')|} = \frac{|\og(q)|}{|\og([a, a^{-1}])| \cdot |\og(q')|}\]
    such subspaces $W$.

    Now we fix $W$ with $q|_W = [a, a^{-1}]$ and count the number of cycles containing $0$ in $W$.
    By the proof of \Cref{lem:triangle}, two vectors $v, w \in W_a = V_{a} \cap W$ form a cycle of length 3 with $0$ if and only if $b_q(v, w) = a$,  i.e. if and only if $v + w \in W_{3a}$.
    For every $u \in W_{3a}$, there are unique $v,w \in W_a$ with $u = v+w$, by \Cref{lem:NumberOfSums} \ref{lem:NumberOfSums2}.
    So each $u \in W_{3a}$ determines a unique cycle of length $3$ containing $0$ within $W$, so there are $|W_{3a}| = |V_{[a, a^{-1}],3a}|$ such.
\end{proof}

The following example illustrates how to use \Cref{prop:CyclesLength3}.

\begin{example}
    We consider the 4-fold sum of squares form $q = \qf{1, 1, 1, 1}$ over $\F_5$ and want to compute the number of cycles of length 3 in $\Gcal_{q, 1}$.
    Note that $q$ is hyperbolic, since $-1 = 4 = 2^2$ is a square in $\F_5$.
    
    We have $c_1 = 0$ since $\cha(\F_5) \neq 3$.
    For $c_2$ both approaches introduced in \Cref{lem:CyclesLength3With0} work, since $\cha(\F_5) = 5$.
    So we consider the isometries 
    \[\qf{1} \perp \qf{1, 1, 1} \cong q \cong [1, 1] \perp [1, 1].\]

    Using the decomposition on the left, we get
    \begin{align*}
        c_2 = \frac12 \cdot (5^3 - 5) \cdot (5^2 - 5) = 1200,
    \end{align*}
    by \Cref{lem:CyclesLength3With0} \ref{lem:CyclesLength3With01} and \Cref{prop:NumberOfRepresentations}.

    For the decomposition on the right we use \Cref{lem:CyclesLength3With0} \ref{lem:CyclesLength3With02}.
    Inserting the orders of the respective orthogonal groups (which can be found, for example, in \cite[Sections 3.7.2, 3.8.2]{book:Wilson}), we get
    \begin{align*}
        c_2 = \frac{(5 + 1) \cdot 2\cdot 5^2 \cdot (5^2 - 1) \cdot (5^2 - 1)}{\left( 2 \cdot (5 + 1)  \right) \cdot \left( 2 \cdot (5 + 1) \right)} = \frac{172800}{144} = 1200,
    \end{align*}
    in accordance with the first calculation.
    
    By \Cref{prop:CyclesLength3} the total number of cycles of length 3 is
    \[\frac{1200 \cdot 5^4}3 = 250000.\]
\end{example}

\subsection{Cycles of Length 4} 
By \Cref{thm:girth} the minimum length of a cycle is $4$ only for $2$-dimensional forms and one exception, the form $\bH \perp[1,1]$ over $\F_2$.
We will now compute the number of $4$-cycles for arbitrary $2$-dimensional forms.
Once we have done that, we will briefly discuss why counting such cycles in higher dimensional forms seems to be a much harder task and compute the number of $4$-cycles of the above $4$-dimensional exception by hand.

We base much of our investigation of cycles $(x_0, x_1, x_2, x_3, x_0)$ on their \emph{diagonal vectors} $x_2 - x_0$ and $x_3 - x_1$, which in many cases determine the cycle itself.
A first useful observation is that if $(0,u,w,v,0)$ is a cycle of length $4$, then we have two distinct $0$-$w$-paths by traversing the cycle in opposite directions from $0$ to $w$.
Now if $\dim(q) = 2$, then \Cref{lem:NumberOfSums} implies that these are the only $0$-$w$-paths of length $2$ in $\Gcal_{q,a}$ and 
\begin{equation} \label{eq:unique:decomposition:diagonal}
    w = u+v,
\end{equation}
since every such path gives rise to a decomposition $u + (w - u) = w = (w - v) + v$ with $x, y \in V_a$, and up to order there is only one such decomposition, so $(w - u) = v$ and $(w - v) = u$.

Essentially, this observation is sufficient to compute the number of cycles of length 4 in representation graphs of forms of dimension 2.

\begin{theorem} \label{thm:cycles4anisotropic}
    Let $F$ be a finite field with $f$ elements, $a \in \Dtilde(q)$, and let $q$ be a quadratic form of dimension 2 over $F$.
    The number of cycles of length $4$ in $\Gcal_{q, a}$ is given as in \Cref{tbl:NumberCyclesLength4}.
    \begin{table}[!ht]
    \renewcommand{\arraystretch}{1.5}
        \begin{tabular}{|c|c||c|c|}
            \hline
            condition on $a$ & condition on $q$ & $\cha(F) \neq 2$ & $\cha(F) = 2$\\  
            \hline\hline
            \multirow{2}{*}{$a \neq 0$} & $q$ isotropic & $\frac{f^2\cdot (f-1) \cdot (f-3)}{8}$ & $\frac{f^2\cdot (f-1) \cdot (f-2)}{8}$  \\ \cline{2-4}
             & $q$ anisotropic & $\frac{f^2\cdot (f + 1) \cdot(f - 1)}{8}$ & $\frac{f^3\cdot (f+1)}{8}$  \\
            \hline 
            $a = 0$ & $q$ isotropic & \multicolumn{2}{|c|}{$\frac{f^2}4 \cdot \left(6\cdot \binom{f - 1} 3 + (f - 1)^2\right)$}  \\
            \hline
        \end{tabular}
        \caption{Number of cycles of length 4 for $2$-dimensional forms} \label{tbl:NumberCyclesLength4}
        \end{table}
\end{theorem}
\begin{proof}
    We start with the case $a \neq 0$ by first counting all $4$-cycles containing $0$.
    If $C = (0, u, w, v, 0)$ is such a cycle, then $w \in V_a + V_a$ with exactly two distinct $0$-$w$-paths in $\Gcal_{q,a}$ (see \eqref{eq:unique:decomposition:diagonal}). 
    The number of such $w$ is
    \begin{equation} \label{eq:number:0:4:cycle}
        \frac{|V_a|^2}{2} - |V_a|, \text{ if } \cha(F) \neq 2 \quad \text{and} \quad  \frac{|V_a|\cdot(|V_a|-1)}{2}, \text{ if } \cha(F) = 2
    \end{equation}
    by \Cref{lem:NumberOfSums} \ref{lem:NumberOfSums3} and \ref{lem:NumberOfSums4}.

    Now any cycle can be mapped to such a cycle using a translation by $-x$ where $x$ is one of the four vertices of the cycle.
    So the total number of $4$-cycles is $\frac{|V|}{4}$ times the corresponding value in \eqref{eq:number:0:4:cycle}.
 
    We now turn to the case $a = 0 \in \Dtilde(q)$, so necessarily $q \cong \bH$.
    If $C=(0, u, w, v, 0)$ is a $4$-cycle with $u, v$ linearly dependent, then $\spn(u,v,w)$ is $1$-dimensional and necessarily totally isotropic.
    To see this write $v = \lambda u$ with $\lambda \in F \setminus \{0, 1\}$.
    With \Cref{prop:DiagonalOrthogonal} we have
    \[0 = b_q(w, u - v) = (1 - \lambda) b_q(w, u)\]
    and thus $b_q(w, u) = 0$, which then implies
    \[0 = q(w - u) = q(w) + q(u) - b_q(w, u) = q(w).\]
    If $u, w$ were linearly independent, they would be an orthogonal basis consisting of isotropic vectors, so $q \cong \qf{0, 0}$ would be degenerate, a contradiction.

    Any three pairwise distinct non-zero elements in this subspace give rise to $3!/2 = 3$ distinct $4$-cycles, so a total of $3 \cdot \binom{f-1}{3}$ distinct $4$-cycles within this particular totally isotropic subspace.
    Now there are $2$ such (see \Cref{prop:NumberOfRepresentations}), so in total there are $6 \cdot \binom{f - 1}{3}$ such cycles.

    Let now $(0, u, w, v, 0)$ be a cycle with $u, v$ linearly independent and let $\lambda, \mu \in F$ be such that $w = \lambda u + \mu v$.
    Obviously, we have $b_q(u, v) \neq 0$.
    We also have
    \begin{equation*}
        0 = q(w - u) = (\lambda - 1)\mu b_q(u, v),
    \end{equation*}
    so $\lambda = 1$ or $\mu = 0$ and by symmetry also $\lambda = 0$ or $\mu = 1$.
    Since $w \notin \{0, u, v\}$ we have $\lambda = \mu = 1$ and this choice clearly gives a cycle as desired.
    We have $(f - 1)^2$ possibilities to choose $u, v$ isotropic and linearly independent, and each choice gives a different cycle of length 4, so there are $(f - 1)^2$ such cycles.

    The assertion finally follows from an argument analogous to \Cref{prop:CyclesLength3}.
\end{proof}

Besides the explicit computation of the number of $4$-cycles, we can give some more structural information.
Some of this works in arbitrary dimension $\geq 2$.

\begin{proposition} \label{prop:DiagonalOrthogonal}
    Let $q$ be a quadratic form of dimension $\geq 2$ over an arbitrary field, $a \in F$ and $C$ be a cycle of length 4 with diagonals $d_1, d_2$ in $\Gcal_{q, a}$. 
    \begin{enumerate}
        \item \label{prop:DiagonalOrthogonal1} The diagonals $d_1,d_2$ are perpendicular, i.e. $b_q(d_1, d_2) = 0$.
    \end{enumerate}
    Now let $\dim(q) = 2$ and $a \in F^\ast$. Then also
    \begin{enumerate}[resume]
        \item \label{lem:PossibleDiagLengthCharNot2} If $\cha(F) \neq 2$, we have $q \cong \qf{q(d_1), q(d_2)}$.
        \item \label{lem:PossibleDiagLengthChar2} If $\cha(F) = 2$, we have $d_1 = d_2$.
    \end{enumerate}
\end{proposition}
\begin{proof}
    After a translation, we can assume that $C = (0, u, w, v, 0)$.
    Then
    \[a = q(w - v) = q(w) + q(v) - b_q(w, v) = q(w) + a - b_q(w, v)\]
    and similarly for $w-u$.
    Subtracting these equalities and using the bilinearity of $b_q$ implies $b_q(w, u - v) = 0$, proving \ref{prop:DiagonalOrthogonal1}.

    If $a \neq 0$ and $\cha(F) \neq 2$, then $w$ is not isotropic by \Cref{lem:Va+VaCapV0trivial}.
    This implies that $w,u-v$ are linearly independent, which proves \ref{lem:PossibleDiagLengthCharNot2}.
    For \ref{lem:PossibleDiagLengthChar2} we have $w =u+v = u-v$, by \eqref{eq:unique:decomposition:diagonal} and $\cha(F) = 2$, so both diagonals are identical.
\end{proof}

In general, computing the number of cycles of length 4 in higher dimensions seems to be a much harder task.
In dimension 2, \eqref{eq:unique:decomposition:diagonal} implies that a cycle of the form $(0, u, w, v, 0)$ is already determined by $w$, but in higher dimensions two different cycles can have the same diagonal vector.
This happens e.g. for $F = \F_5, q = \qf{1, 1, 1}, a = 1$ and $w = (1, 4, 0)$, then we have the cycles 
$C_i = (0, u_i, w, v_i, 0)$ where 
\begin{align*}
    u_1 = (0, 4, 0), v_1 = (2, 1, 1)\quad\text{ and }\quad u_2 = (0, 4, 0), v_2 = (1, 0, 0).
\end{align*}
Nevertheless, we wanted to compute the number of cycles of length 4 at least in those cases where cycles of length 4 are the shortest cycles occurring in $\Gcal_{q, a}$, and these cases are covered by our results with only one exception, the form $\bH \perp [1,1]$ over $\F_2$ and $a=1$.
We discuss this briefly as a concluding example to illustrate the greater complexity in higher dimensions.

First, we count the $4$-cycles containing the origin.
We observe (by explicit enumeration, which we omit), that there are $15$ possible choices for the vertex $w$ opposing the origin, fixing one of the diagonals.
For each of these $15$ choices of $w$ there are $6$ distinct $0$-$w$-paths of length $2$ in the graph. 
This gives a total of $15 \cdot \binom{6}{2} = 225$ cycles of length $4$ containing the origin, and a total of $\frac{|\F_2|^4}{4} \cdot 225 = 900$ cycles of length $4$.

\printbibliography
\end{document}